\documentclass[11pt,a4paper]{article}
\usepackage{amsmath, amsfonts, amsthm, amssymb, color}
\usepackage{enumerate}
\usepackage{graphicx}
\usepackage{tikz}

\usetikzlibrary{intersections,calc,arrows.meta}

\newtheorem{thm}{Theorem}

\newtheorem{cor}{Corollary}

\bfseries\normalfont

\begin{document}
\title{Note on the thickness of the Cartesian product of a complete graph and a path}

\author{
Kenta Noguchi\thanks{Department of Information Science and Technology, Tokyo University of Science, Noda, Japan.
Email: {\tt noguchi@rs.tus.ac.jp}}
}

\date{}
\maketitle

\noindent
\begin{abstract}
We determine the thickness of the Cartesian product $K_{6p+4} \, \square \, P_2$ for $p \ge 0$ and of the Cartesian product $K_8 \, \square \, P_m$ for $m \ge 1$, where $K_n$ and $P_m$ denote the complete graph on $n$ vertices and the path on $m$ vertices, respectively. These are two of the three remaining cases from the results of Chen, Kohonen, and Yang (2026). Combined with their results, the thickness of $K_n \, \square \, P_2$ is completely determined.
\end{abstract}


\section{Introduction}
To find a good decomposition of graphs is widely studied in the literature. In this paper, we deal with one of them, called the thickness. For a graph $G$, the \textit{thickness} $\theta(G)$ is defined as the minimum integer $k$ such that there exists an edge decomposition $E(G)=A_1 \cup \cdots \cup A_k$ so that $A_i$ induces a planar graph for $1 \le i \le k$. The thickness of some well-known graph families has been investigated, for example, for the complete graphs~\cite{AG, BH} and for the complete bipartite graphs~\cite{BHM}. Especially, the thickness of the complete graph $K_n$ is completely determined as follows.

\begin{thm}[\cite{AG}, \cite{BH}]
\label{Kn}
The thickness of $K_n$ is $\theta(K_n) = \left\lfloor \frac{n+7}{6} \right\rfloor$, except that $\theta(K_9)=\theta(K_{10})=3$.
\end{thm}

For graphs $G$ and $H$, the Cartesian product of them, denoted by $G\, \square \, H$, is defined as:
\begin{align*}
V(G\, \square \, H) &= V(G) \times V(H),\\
E(G\, \square \, H) &= \left\{ \{(u, v), (u, v')\} \mid u\in V(G), vv'\in E(H) \right\} \cup \left\{ \{(u, v), (u', v)\} \mid uu'\in E(G), v\in V(H) \right\}.
\end{align*}

Chen and Yin \cite{CY}, Yang and Chen \cite{YC}, and Guo and Yang \cite{GY} investigated the thickness of the Cartesian product of graphs. Especially, Yang and Chen \cite{YC} considered the thickness of the Cartesian product of the complete graph $K_n$ and the path $P_m$, and claimed that the following statement is true.

\medskip
\noindent
\textbf{(false) Theorem~A} (Theorems~16 and 17 in \cite{YC})
\textit{The thickness of the Cartesian product $K_n \, \square \, P_m~(n \ge 2, m \ge 2)$ is
\[
\theta(K_n \, \square \, P_m) = \left\lfloor \frac{n+8}{6} \right\rfloor,
\]
except that $\theta(K_8 \, \square \, P_m) = \theta(K_9 \, \square \, P_m) = 3~(m \ge 2)$ and possibly when $n=6p+3~(p \ge 2)$ with $m \ge 3$.}
%

\medskip
However, unfortunately there is an error in the proof. In fact, the statement is false as Chen, Kohonen, and Yang~\cite{CKY} showed $\theta(K_8 \, \square \, P_2) = \theta(K_8 \, \square \, P_3) = 2$. Instead, they showed the following theorems.

\begin{thm}[Theorem~2 in \cite{CKY}] 
\label{prevthm2}
The thickness of the Cartesian product $K_n \, \square \, P_2~(n \ge 2)$ is
\[
\theta(K_n \, \square \, P_2) = \left\lfloor \frac{n+8}{6} \right\rfloor,
\]
except that $\theta(K_9 \, \square \, P_2) = 3$ and possibly when $n=6p+4~(p \ge 2)$.
\end{thm}

\begin{thm}[Theorem~3 in \cite{CKY}] 
\label{prevthm3}
The thickness of the Cartesian product $K_n \, \square \, P_m~(n \ge 2, m \ge 3)$ is
\[
\theta(K_n \, \square \, P_m) = \left\lfloor \frac{n+9}{6} \right\rfloor,
\]
except that $\theta(K_3 \, \square \, P_m) = 1$ and possibly when $n=6p+3, 6p+4$ and $n=8~(p \ge 2)$. Moreover, $\theta(K_8 \, \square \, P_3) = 2$.
\end{thm}

In the present paper, we show the following theorems.

\begin{thm} 
\label{mainthm1}
Let $p \ge 0$ be an integer. The thickness of the Cartesian product $K_{6p+4} \, \square \, P_2$ is
\[
\theta(K_{6p+4} \, \square \, P_2) = p+2.
\]
\end{thm}

\begin{thm} 
\label{mainthm2}
Let $m \ge 1$ be an integer. The thickness of the Cartesian product $K_8 \, \square \, P_m$ is
\[
\theta(K_8 \, \square \, P_m) = 2.
\]
\end{thm}

Combining Theorems~\ref{mainthm1}, \ref{mainthm2}, \ref{prevthm2}, and \ref{prevthm3}, we have the following corollary, which corresponds to the true version of Theorem~A.
Note that the thickness of $K_n \, \square \, P_2$ is completely determined.

\begin{cor}
\label{cor1}
Let $n \ge 1$ and $m \ge 2$ be integers. The thickness of the Cartesian product $K_n \, \square \, P_m$ is
\[
\theta(K_n \, \square \, P_m) = \left\lfloor \frac{n+8}{6} \right\rfloor,
\]
except that $\theta(K_9 \, \square \, P_m) = 3~(m \ge 2)$ and possibly when $n=6p+3~(p \ge 2)$ with $m \ge 3$.
\end{cor}

\section{Preliminary}
\label{Sec2}
We refer to the basic terminology in \cite{BM}. In this paper all graphs are simple. Recall that $K_n$ and $P_m$ denote the complete graph on $n$ vertices and the path on $m$ vertices, respectively. For the thickness, we often use the fact that if $H$ is a subgraph of $G$, then $\theta(H) \le \theta(G)$. For a graph $G$ with thickness $k$, we call a family of planar graphs $\{ G_1, \ldots, G_k \}$ a \textit{planar decomposition} if $V(G_1) = \cdots = V(G_k) = V(G)$ and $E(G_1) \cup \cdots \cup E(G_k) = E(G)$ is an edge decomposition. When $k=2$, we also call a pair of planar graphs $G_1$ and $G_2$ a \textit{biplanar decomposition}.

We note some results on the Cartesian product $K_n \, \square \, P_m$. In 1977, Ringel \cite{Ri} considered the genus of $K_n \, \square \, P_2$ and determined it for about five-sixths of all values $n$. Recently, Sun \cite{Sun} completely determined it for all $n$. Badgett, Millichap, and the author \cite{BMN} determined the genus of $K_4 \, \square \, P_3$, to characterize the toroidal Cartesian product where one factor is 3-connected.

\section{Proofs of the main theorems}
\label{Sec3}
%
First, we prove Theorem~\ref{mainthm1}.
For the graph $K_n \, \square \, P_2$, we call an edge \textit{$P_2$-edge} if it joins different $K_n$'s vertices.

\medskip
\begin{proof}[Proof of Theorem \ref{mainthm1}]
By the inequality~(2) in \cite{YC} and Theorem~\ref{Kn}, we have $\theta(K_{6p+4} \, \square \, P_2) \le \theta(K_{6p+5}) = p+2$. We now show that $\theta(K_{6p+4} \, \square \, P_2) \ge p+2$. 
Let $n \ge 2$, $\theta:= \theta(K_n \, \square \, P_2)$, and $\{ G_1, \ldots, G_{\theta} \}$ be a planar decomposition of $K_n \, \square \, P_2$. 
For $1 \le i \le \theta$, consider some fixed planar embedding of $G_i$, and let $T_i$ be a maximal plane supergraph of it on the same $2n$ veritces. It is well-known that $T_i$ is a triangulation and $|E(T_i)|=3|V(T_i)|-6=6n-6$.
Let \(x_i\) be the number of \(P_2\)-edges contained in \(G_i\). Then
\[ \sum_{i=1}^{\theta}x_i=n \]
and
\[ \sum_{i=1}^{\theta}|E(G_i)|=|E(K_n \, \square \, P_2)|=n^2. \]
Let
$a_i=|E(T_i)\setminus E(G_i)| =6n-6-|E(G_i)|$. 
Now consider a \(P_2\)-edge \(e\) of \(G_i\). Such an edge belongs to no triangle of \(K_n\, \square \, P_2\). In the triangulation \(T_i\), the edge \(e\) is incident with two triangular faces, and each of these two triangles must therefore contain at least one newly added edge.
Moreover, since the \(P_2\)-edges do not share a vertex, a triangular face contains at most one \(P_2\)-edge. Since each newly added edge is incident with two triangular faces, counting incidences between \(P_2\)-edges and their incident triangular faces gives
\[ 2x_i\le2a_i. \]
Consequently,
\[ |E(G_i)|+x_i \le |E(G_i)|+a_i=6n-6. \]
Summing over \(i\),
\[ n^2+n = \sum_{i=1}^{\theta}\bigl(|E(G_i)|+x_i\bigr) \le \theta(6n-6). \]
Thus,
\[ \theta\ge \frac{n(n+1)}{6(n-1)} > \frac{n+2}{6}. \]
Substituting $n=6p+4$, we have $\theta > p+1$, as desired.
\end{proof}

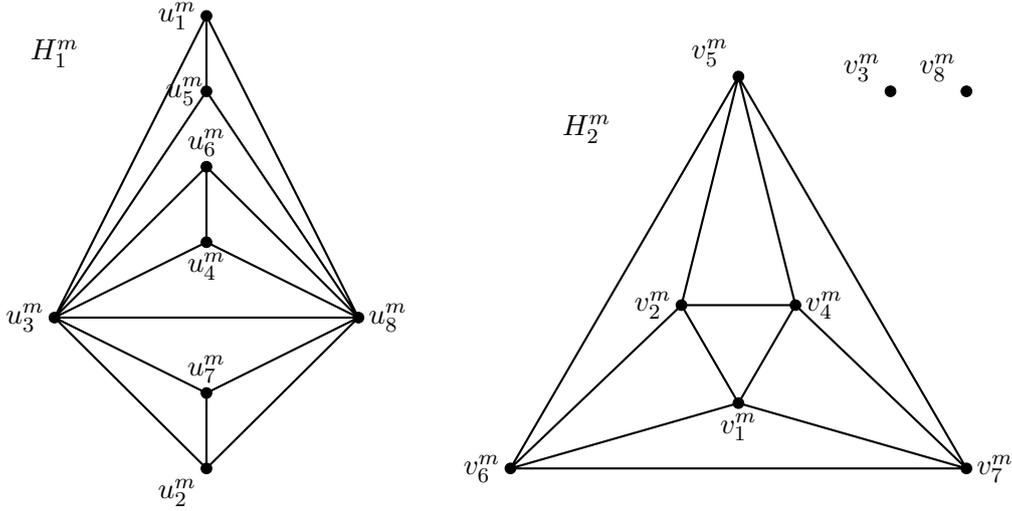
\begin{figure}[t]
 \centering
\begin{tikzpicture}
\draw (-2,3.5) node{$H_1^m$};
\draw[fill=black] (0,4) circle (2pt) node[left]{$u_1^m$};
\draw[fill=black] (0,3) circle (2pt);
\draw[fill=black] (0.1,3) node[left]{$u_5^m$};
\draw[fill=black] (0,2) circle (2pt) node[above]{$u_6^m$};
\draw[fill=black] (0,1) circle (2pt) node[below]{$u_4^m$};
\draw[fill=black] (0,-1) circle (2pt) node[above]{$u_7^m$};
\draw[fill=black] (0,-2) circle (2pt) node[below left]{$u_2^m$};
\draw[fill=black] (-2,0) circle (2pt) node[left]{$u_3^m$};
\draw[fill=black] (2,0) circle (2pt) node[right]{$u_8^m$};
\draw[thick] (2,0)--(0,4)--(-2,0)--(0,3)--(2,0)--(0,2)--(-2,0)--(0,1)--(2,0)--(-2,0)--(0,-1)--(2,0)--(0,-2)--(-2,0);
\draw[thick] (0,4)--(0,3);
\draw[thick] (0,2)--(0,1);
\draw[thick] (0,-1)--(0,-2);

\def\x{7}
\def\y{-2}
\draw (\x-2,2.5) node{$H_2^m$};
\draw[fill=black] (\x+2,\y+5) circle (2pt) node[above left]{$v_3^m$};
\draw[fill=black] (\x+3,\y+5) circle (2pt) node[above left]{$v_8^m$};
\draw[fill=black] (\x+0,\y+5.196) circle (2pt) node[above left]{$v_5^m$};
\draw[fill=black] (\x-0.75,\y+2.165) circle (2pt) node[left]{$v_2^m$};
\draw[fill=black] (\x+0.75,\y+2.165) circle (2pt) node[right]{$v_4^m$};
\draw[fill=black] (\x+0,\y+0.866) circle (2pt) node[below]{$v_1^m$};
\draw[fill=black] (\x-3,\y) circle (2pt) node[left]{$v_6^m$};
\draw[fill=black] (\x+3,\y) circle (2pt) node[right]{$v_7^m$};
\draw[thick] (\x+0,\y+5.196)--(\x-3,\y)--(\x+3,\y)--(\x+0,\y+5.196)--(\x-0.75,\y+2.165)--(\x-3,\y)--(\x+0,\y+0.866)--(\x-0.75,\y+2.165)--(\x+0.75,\y+2.165)--(\x+0,\y+0.866)--(\x+3,\y)--(\x+0.75,\y+2.165)--(\x+0,\y+5.196);
\end{tikzpicture}
 \caption{Planar graphs $H_1^m$ and $H_2^m$.}
 \label{fig:1}
\end{figure}

\medskip
\begin{proof}[Proof of Theorem \ref{mainthm2}]
By Theorem~\ref{Kn} and the fact $K_8 \, \square \, P_m \supseteq K_8$, we have $\theta(K_8 \, \square \, P_m) \ge \theta(K_8) = 2$. We now show that $\theta(K_8 \, \square \, P_m) \le 2$ for all $m\ge 1$, by constructing a biplanar decomposition $\{ G_1^m, G_2^m \}$ of $K_8 \, \square \, P_m$ recursively. Note that for $m \in \{2, 3\}$, Chen, Kohonen, and Yang~\cite[Figures~1 and 2]{CKY} has already constructed them. 
For $m=1$, let $G_1^1 = H_1^m$ and $G_2^1 = H_2^m$ where $H_1^m$ and $H_2^m$ are depicted in Figure~\ref{fig:1}. (Assuming $u_i^m = v_i^m$ for $1 \le i \le 8$, we see that the pair $G_1^1$ and $G_2^1$ is a biplanar decomposition of $K_8 \simeq K_8 \, \square \, P_1$.) Let $m \ge 2$ and assume that we have a biplanar decomposition $\{ G_1^{m-1}, G_2^{m-1} \}$ of $K_8 \, \square \, P_{m-1}$ that is constructed along this proof. Now, for convenience, let $V(G_1^{m-1}) = \bigcup_{j=1}^{m-1} \{u_1^j, \ldots, u_8^j\}$ and $V(G_2^{m-1}) = \bigcup_{j=1}^{m-1} \{v_1^j, \ldots, v_8^j\}$ so that identifying $u_i^j$ and $v_i^j$ for $1 \le i \le 8$ and $1 \le j \le m-1$ results in the graph $K_8 \, \square \, P_{m-1}$.

If $m$ is even, then prepare the graphs $I_1^m$ and $I_2^m$ depicted in Figure~\ref{fig:2}.
(a-1) Insert $G_1^{m-1}$ into the triangular face $u_1^mu_2^mu_3^m$ in $I_1^m \setminus \{ u_4^m\}$ and add the three edges $u_1^{m-1}u_1^m, u_2^{m-1}u_2^m, u_3^{m-1}u_3^m$, and
(a-2) insert the vertex $u_4^m$ into $G_1^{m-1}$ and add the edge $u_4^{m-1}u_4^m$; let the resulting graph be $G_1^m$.
(b-1) Insert $G_2^{m-1} \setminus \{ v_8^{m-1}\}$ into the quadrilateral face $v_5^mv_4^mv_6^mv_7^m$ in $I_2^m$ and add the three edges $v_5^{m-1}v_5^m, v_6^{m-1}v_6^m, v_7^{m-1}v_7^m$, and
(b-2) insert the vertex $v_8^{m-1}$ into $I_2^m$ and add the edge $v_8^{m-1}v_8^m$; let the resulting graph be $G_2^m$.

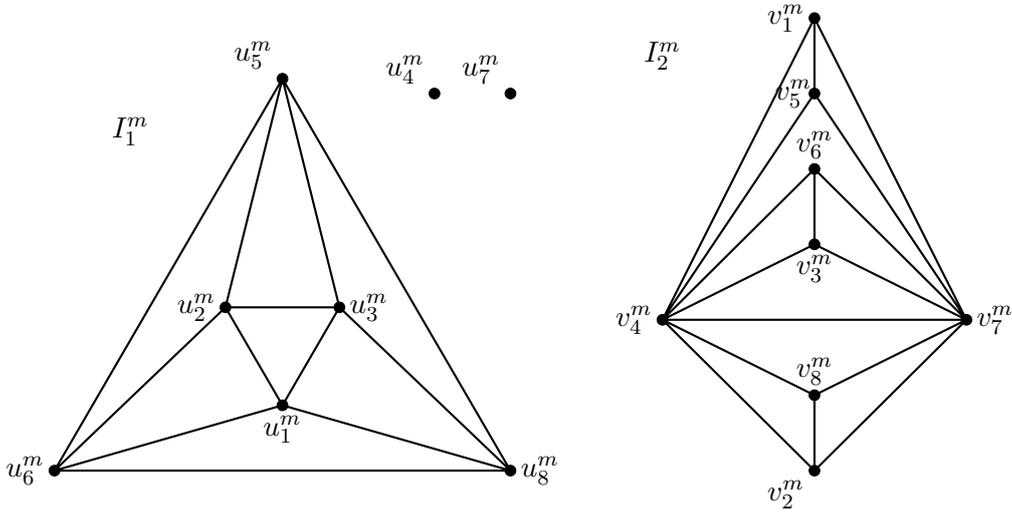
\begin{figure}[t]
 \centering
\begin{tikzpicture}
\def\x{-7}
\def\y{-2}
\draw (\x-2,2.5) node{$I_1^m$};
\draw[fill=black] (\x+2,\y+5) circle (2pt) node[above left]{$u_4^m$};
\draw[fill=black] (\x+3,\y+5) circle (2pt) node[above left]{$u_7^m$};
\draw[fill=black] (\x+0,\y+5.196) circle (2pt) node[above left]{$u_5^m$};
\draw[fill=black] (\x-0.75,\y+2.165) circle (2pt) node[left]{$u_2^m$};
\draw[fill=black] (\x+0.75,\y+2.165) circle (2pt) node[right]{$u_3^m$};
\draw[fill=black] (\x+0,\y+0.866) circle (2pt) node[below]{$u_1^m$};
\draw[fill=black] (\x-3,\y) circle (2pt) node[left]{$u_6^m$};
\draw[fill=black] (\x+3,\y) circle (2pt) node[right]{$u_8^m$};
\draw[thick] (\x+0,\y+5.196)--(\x-3,\y)--(\x+3,\y)--(\x+0,\y+5.196)--(\x-0.75,\y+2.165)--(\x-3,\y)--(\x+0,\y+0.866)--(\x-0.75,\y+2.165)--(\x+0.75,\y+2.165)--(\x+0,\y+0.866)--(\x+3,\y)--(\x+0.75,\y+2.165)--(\x+0,\y+5.196);

\draw (-2,3.5) node{$I_2^m$};
\draw[fill=black] (0,4) circle (2pt) node[left]{$v_1^m$};
\draw[fill=black] (0,3) circle (2pt);
\draw[fill=black] (0.1,3) node[left]{$v_5^m$};
\draw[fill=black] (0,2) circle (2pt) node[above]{$v_6^m$};
\draw[fill=black] (0,1) circle (2pt) node[below]{$v_3^m$};
\draw[fill=black] (0,-1) circle (2pt) node[above]{$v_8^m$};
\draw[fill=black] (0,-2) circle (2pt) node[below left]{$v_2^m$};
\draw[fill=black] (-2,0) circle (2pt) node[left]{$v_4^m$};
\draw[fill=black] (2,0) circle (2pt) node[right]{$v_7^m$};
\draw[thick] (2,0)--(0,4)--(-2,0)--(0,3)--(2,0)--(0,2)--(-2,0)--(0,1)--(2,0)--(-2,0)--(0,-1)--(2,0)--(0,-2)--(-2,0);
\draw[thick] (0,4)--(0,3);
\draw[thick] (0,2)--(0,1);
\draw[thick] (0,-1)--(0,-2);
\end{tikzpicture}
 \caption{Planar graphs $I_1^m$ and $I_2^m$.}
 \label{fig:2}
\end{figure}

If $m$ is odd, then prepare the graphs $H_1^m$ and $H_2^m$ depicted in Figure~\ref{fig:1}.
(a-1)~insert $G_1^{m-1} \setminus \{ u_7^{m-1}\}$ into the quadrilateral face $u_5^mu_3^mu_6^mu_8^m$ in $H_1^m$ and add the three edges $u_5^{m-1}u_5^m, u_6^{m-1}u_6^m$, $u_8^{m-1}u_8^m$, and
(a-2)~insert the vertex $u_7^{m-1}$ into $H_1^m$ and add the edge $u_7^{m-1}u_7^m$; let the resulting graph be $G_1^m$.
(b-1)~Insert $G_2^{m-1}$ into the triangular face $v_1^mv_2^mv_4^m$ in $H_2^m \setminus \{ v_3^m\}$ and add the three edges $v_1^{m-1}v_1^m, v_2^{m-1}v_2^m, v_4^{m-1}v_4^m$, and
(b-2)~insert the vertex $v_3^m$ into $G_2^{m-1}$ and add the edge $v_3^{m-1}v_3^m$; let the resulting graph be $G_2^m$.

In both cases, it is straightforward to check that the pair $G_1^m$ and $G_2^m$ is a biplanar decomposition of $K_8 \, \square \, P_m$. Notice that, for every even $m$, the outer face of $G_1^m$ is bounded by $u_5^mu_6^mu_8^m$ (and the isolated vertex $u_7^m$) and the outer face of $G_2^m$ is bounded by $v_1^mv_4^mv_2^mv_7^m$; for every odd $m$, the outer face of $G_1^m$ is bounded by $u_1^mu_3^mu_2^mu_8^m$ and the outer face of $G_2^m$ is bounded by $v_5^mv_6^mv_7^m$ (and the isolated vertex $v_8^m$) so that the proof works.
\end{proof}

\medskip
\begin{proof}[Proof of Corollary \ref{cor1}]
First, consider the case $m=2$. The case $n=1$ obviously follows. The case $n=6p+4$ follows from Theorem~\ref{mainthm1} since $\theta(K_{6p+4} \, \square \, P_2) = p+2 = \left\lfloor \frac{(6p+4)+8}{6} \right\rfloor$. The other cases follow from Theorem~\ref{prevthm2}. 
Second, consider the case $m \ge 3$. The case $n=1$ obviously follows. The case $n=8$ follows from Theorem~\ref{mainthm2} since $\theta(K_8 \, \square \, P_m) = 2 = \left\lfloor \frac{8+8}{6} \right\rfloor$. For $n=6p+4~(p \ge 0)$, it has already been shown by the inequality~(3) in \cite{YC} and Theorem~\ref{Kn} that $\theta(K_{6p+4} \, \square \, P_m) \le \theta(K_{(6p+4)+2}) = p+2 = \left\lfloor \frac{(6p+4)+8}{6} \right\rfloor$. Additionally, by Theorem~\ref{mainthm1}, $\theta(K_{6p+4} \, \square \, P_m) \ge \theta(K_{6p+4} \, \square \, P_2) = p+2$. Hence, the case $n=6p+4$ follows for all $p \ge 0$. The other cases follow from Theorem~\ref{prevthm3}.
\end{proof}

\section{Concluding remarks}
On Corollary \ref{cor1}, for the remaining case $n=6p+3$ with $m \ge 3$, deciding whether $\theta(K_{6p+3} \, \square \, P_m) = p+1$ or $p+2$ seems to be difficult by the following reason. As proven in \cite{YC}, the truth of $\theta(K_{6p+5}-e) = p+1$, which is conjectured by Hobbs \cite{Ho} for $p \ge 3$, would imply that $\theta(K_{6p+3} \, \square \, P_m) = p+1 = \left\lfloor \frac{(6p+3)+8}{6} \right\rfloor$. However, proving $\theta(K_{6p+5}-e) = p+1$ is very challenging since it is much harder than proving $\theta(K_{6p+4}) = p+1$, whose proof is already much involved (see~\cite{AG}). Unlike the case $n=8$, we feel that it is hard to prove $\theta(K_{6p+3} \, \square \, P_m) \le p+1$ without using the assumption of $\theta(K_{6p+5}-e) = p+1$.

\section*{Acknowledgements}
I greatly appreciate one of the anonymous referees who not only pointed out an error in the proof of Theorem~4 in the previous version of this paper, but also kindly gave an alternative shorter proof.

\bigskip
\section*{Funding}
This research was partially supported by JSPS KAKENHI Grant Number JP21K13831. 

\medskip
\section*{Availability of data and materials}
Not applicable.

\medskip
\section*{Competing Interests}
The author has no relevant financial or non-financial interests to disclose.

\begin{thebibliography}{99}
\bibitem{AG}
V. B. Alekseev and V. S. Gon\v{c}hakov,
The thickness of an arbitrary complete graph,
Math. Sbornik. \textbf{30} (1976), 187--202.

\bibitem{BMN}
E. Badgett, C. Millichap, and K. Noguchi,
Toroidal Cartesian products where one factor is $3$-connected,
Graphs Combin. \textbf{42} (2026), Paper No.~49, 14 pp.

\bibitem{BH}
L. W. Beineke and F. Harary,
The thickness of the complete graph,
Canadian J. Math. \textbf{17} (1965), 850--859.

\bibitem{BHM}
L. W. Beineke, F. Harary, and J. W. Moon,
On the thickness of the complete bipartite graph,
Math. Proc. Cambridge Philos. Soc. \textbf{60} (1964), 1--5.

\bibitem{BM}
J. A. Bondy, U. S. R. Murty, 
Graph Theory, Graduate Texts in Mathematics 244, 
Springer, New York, 2008.

\bibitem{CKY}
Y. Chen, J. Kohonen, and Y. Yang,
Erratum to ``The thickness of amalgamations and Cartesian product of graphs'',
Discuss. Math. Graph Theory \textbf{46} (2026), 91--95.

\bibitem{CY}
Y. Chen and X. Yin,
The thickness of the Cartesian product of two graphs,
Canad. Math. Bull. \textbf{59} (2016), 705--720.

\bibitem{GY}
X. Guo and Y. Yang,
The thickness of some Cartesian product graphs,
Ars Combin. \textbf{147} (2019), 97--107.

\bibitem{Ho}
A. M. Hobbs,
A survey of thickness,
in: Recent Progress in Combinatorics (Proc. 3rd Waterloo Conf. on Combinatorics, 1968), W. T. Tutte (Ed(s)), (New York, Academic Press, 1969) 255--264.

\bibitem{Ri}
G. Ringel,
On the genus of the graph $K_n \times K_2$ or the $n$-prism,
Discrete Math. \textbf{20} (1977), 287--294.

\bibitem{Sun}
T. Sun,
Settling the genus of the $n$-prism,
European J. Combin. \textbf{110} (2023), Paper No.~103667, 12 pp.

\bibitem{YC}
Y. Yang and Y. Chen,
The thickness of amalgamations and Cartesian product of graphs,
Discuss. Math. Graph Theory \textbf{37} (2017), 561--572.
\end{thebibliography}
\end{document}